\documentclass[a4paper,11pt]{amsart}

\usepackage{latexsym}
\usepackage{amssymb,amsmath, amsthm}
\usepackage{graphics}
\usepackage{url}
\usepackage[vcentermath]{youngtab}
  
\usepackage{booktabs}  

\newtheorem{theorem}{Theorem} 

\newtheorem{proposition}[theorem]{Proposition}

\newtheorem{lemma}[theorem]{Lemma}

\theoremstyle{definition}

\newcommand{\N}{\mathbf{N}}
\newcommand{\Z}{\mathbf{Z}}

\newcommand{\SW}{SW}

\renewcommand{\epsilon}{\varepsilon}

\newcounter{thmlistcnt}
	{\setcounter{thmlistcnt}{0}%
	\begin{list}{\emph{(\roman{thmlistcnt})}}{%
		\usecounter{thmlistcnt}%
		\setlength{\topsep}{0pt}%
		\setlength{\leftmargin}{36pt}%
		\setlength{\labelwidth}{17pt}%
		\setlength{\itemsep}{0pt}%
		\setlength{\itemindent}{0pt}}%
	}%
	{\end{list}}%

\newcommand{\wt}[1]{||#1||}

\subjclass[2010]{Primary 91A46, Secondary 05C57, 94A50}

\begin{document}
\title[The Majority Game]{The majority game with an arbitrary majority}
\author{John R. Britnell and Mark Wildon}
\date{\today}
\maketitle
\thispagestyle{empty}

\begin{abstract}
The $k$-majority game is played with~$n$ numbered balls, each coloured with one of two colours. It is given
that there are at least~$k$ balls of the majority colour, where $k$ is a fixed integer greater than $n/2$.
On each turn the player selects two balls to compare,
and it is revealed whether they are of the same colour; the player's aim is to determine a ball of the majority colour.
It has been correctly stated by Aigner that the minimum number of comparisons necessary to guarantee success is $2(n-k) - B(n-k)$, where $B(m)$ is the weight of the 
binary expansion of $m$. However his proof contains an error. We give an alternative proof of this result, which generalizes an argument of Saks and Werman.
\end{abstract}

\section{Introduction}

Fix $n$ and $k \in \N$ with $k > n/2$. The $k$-majority game is played with~$n$ 
numbered balls
which are each coloured with one of two colours. It is given
that there are at least~$k$ balls of the majority colour.
On each turn the player selects two balls to compare,
and it is revealed whether they are of the same colour, or of 
different colours. The player's objective is to determine a ball of the majority colour.
We write $K(n,k)$ for the minimum number of comparisons that will guarantee success. 
We write $B(m)$ for the number of digits $1$ in the binary representation of $m \in \N_0$.
The object of this paper is to prove the following theorem.

\begin{theorem}\label{thm:main}
If $n$, $k \in \N$ and $k \le n$ then $K(n,k) = 2(n-k) - B(n-k)$.
\end{theorem}

This theorem has been stated previously, as Theorem 3 of \cite[page 14]{Aigner}. We believe, however, that there is a flaw in the proof offered there
of the lower bound for $K(n,k)$, i.e.~the fact that $2(n-k) - B(n-k)$ comparisons are necessary.
The error arises in Case (ii) of the proof of Lemma 1, in which 
it is implicitly assumed that if it is optimal at some point 
for the player to compare balls~$i$ and~$t$, then there exist two balls~$j$ and~$\ell$ which it is optimal
to compare on the next turn, irrespective
of the answer received when balls $i$ and~$t$ are compared. The proof of Theorem~3 requires an analogue of 
Lemma~1, stated as Lemma~3, which inherits the same error. The argument for Lemma~1 of \cite{Aigner}
is based on Lemma~5.1 in~\cite{Wiener}, which contains the same flaw; the authors are grateful to Prof.~Aigner and Prof.~Wiener for confirming these errors.\footnote{Personal communications.}

In \cite{SaksWerman}, Saks and Werman 
have shown that
$K(2m+1,m+1) = 2m - B(m)$.
(An independent proof, using an elegant argument on the game tree,
was later given by Alonso, Reingold and Schott \cite{AlonsoEtAl}.) 
The original contribution of this paper is
to supply a correct proof that $2(n-k) - B(n-k)$ questions are necessary in the general case, by generalizing the argument of Saks and Werman~\cite{SaksWerman}. 




We remark that an alternative setting for the majority problem replaces the $n$ balls
with a room of $n$ people. Each person
is either a knight, who always tells the truth, or a knave, who always lies.
The question `Person $i$, is person $j$ a knight?'
corresponds to a comparison between balls $i$ and $j$. (The asymmetry in the 
form of the questions is therefore illusory.)
In \cite[Theorem 6]{Aigner}, Aigner gives a clever questioning
strategy which demonstrates that $2(n-k) - B(n-k)$ questions
suffice, even when knaves are replaced by spies (Aigner's unreliable people), who may answer as they see fit.
He subsequently uses his Lemma 3 to show that $2(n-k) - B(n-k)$ questions are also necessary;
our Theorem~\ref{thm:main} can be used to replace this lemma, and so repair
the gap in the proof of Theorem 6 of \cite{Aigner}.

We refer the reader to \cite{Aigner} and the recent preprint \cite{ChengEtAl} 
for a number of results on further questions that arise in this setting.

\section{Preliminary reformulation}
We begin with a standard reformulation of the problem 
that follows \hbox{\cite[\S 2]{Aigner}}. In the special case $n=2m+1$ and $k=m+1$, it
may also be found in \cite[\S 4]{SaksWerman} and \cite[\S 2, \S3]{Wiener}.
A position in a $k$-majority game corresponds to a graph on $n$ vertices, in
which there is an edge, labelled either `same' or `different', between vertices $i$ and $j$ if balls $i$ and $j$
have been compared. Each connected component of this graph
admits a unique bipartition into parts corresponding to balls
of the same colour. 
If $C$ is a component with bipartition $\{X, Y\}$
where $|X| \ge |Y|$ then we define the \emph{weight} of $C$ to be $|X| - |Y|$.

Suppose that the graph has distinct components $C$ and $C'$ of weights $w$ and $w'$ respectively,
where $w \ge w'$,
and that balls in $C$ and $C'$ are compared.
In the new question graph, 
$C$ and~$C'$ are united in a single component; it is easily checked that
this component has  weight either $w+w'$ or $w-w'$, depending
on which parts of $C$ and~$C'$ the balls lie in and which answer is given.
Moreover, if a game position has components
of weights $w_1, \ldots, w_c$ where $w_1 \ge w_2 \ge \ldots \ge w_c$
then exactly $n-c$ comparisons have been made.
Finally, if $e = k-(n-k)$ is the minimum
excess of the majority colour over the minority colour,
then $w_1 + \cdots + w_c = 2s + e$ for some $s \in \N_0$ and
the balls in the larger part of the component of weight $w_i$
can consistently be of the minority colour if and only if
$w_1 \le s$. (This is an equivalent condition to equation (14) in \cite{Aigner}.)




These remarks show that the $k$-majority game can be reformulated
as a two player adversarial game 
played on multisets of non-negative
integers, which we shall call \emph{positions}. As above, let $e = k - (n-k)$.
The starting
position is the multiset $\{1,\ldots, 1\}$ containing $n$ elements. In 
each turn, two distinct multiset elements $w$ and $w'$, with $w \ge w'$, 
are chosen by the \emph{Selector}, and the \emph{Assigner}
chooses to replace them with either $w+w'$ or $w-w'$.
The game ends as soon as a position $\{w_1,\ldots, w_c\}$ is reached such
that $w_1 \ge \ldots \ge w_c$ and
\[ w_1 \ge s+1, \]
where $s$ is determined by $2s+e = w_1+\cdots + w_c$.
Following \cite{SaksWerman}, we call
such positions \emph{final}. We
define the \emph{value} of a general position $M$ to be the
number of elements in  a final position reached from $M$, assuming,
as ever, optimal play by both sides. We denote the value of $M$ by $V(M)$.

The result we require, that $2(n-k) - B(n-k)$ questions are
necessary to identify a ball of the majority colour
in the $k$-majority game, is equivalent
to the following proposition.

\begin{proposition}\label{prop:bound}
Let $n \in \N$ and let $k > n/2$. The value of the 
starting position in the $k$-majority game
is at most $B(n-k) + k - (n-k)$.
\end{proposition}


\section{Generalized Saks--Werman statistics}

If $M$ is a position in a majority game and $N$ is a submultiset
of $M$ then we shall say that $N$ is a \emph{subposition} of $M$.
Let $\bar{N}$ denote the complement of $N$ in $M$ and let
$\wt{M}$ denote the sum of all the elements of $N$.
Let $\epsilon_M(N) = \wt{\bar{N}} - \wt{N}$.
For $e \in \N$ and a position $M$ such that $\wt{M}$ and 
$e$ have the same parity, we define 
\[ \delta_e(M) = \sum_N (-1)^{\wt{N}} \]
where the sum is over all subpositions $N$ of $M$ such that 
$\epsilon_M(N) \ge e$.
Thus a subposition of $M$ contributes to $\delta_e(M)$ if and only
if it corresponds to a colouring of the balls
in which the
excess of the majority colour over the minority colour is at least $e$.
We note that when $e=1$ 
we have
$\delta_1(M) = -f_M(-1)$ where $f_M$
is the polynomial defined in \cite[page 386]{SaksWerman}. 
(The reason
for working with 
minority subpositions rather than majority subpositions, as in \cite{SaksWerman},
will be seen in the proof of Lemma~\ref{lemma:alt}.)




The following lemma is a generalization of \cite[Lemma 4.2]{SaksWerman}.


\begin{lemma}\label{lemma:conserved}
Let $M$ be a position and let $e$ have the same parity as $\wt{M}$.
Let $w, w' \in M$ be two elements of $M$ with $w \ge w'$.
Let $M^+$ and $M^-$ be the positions obtained from $M$ if $w$ and $w'$
are replaced with $w+w'$ and $w-w'$, respectively. Then
\[ \delta_e(M) = \delta_e(M^+) + (-1)^{w'} \delta_e(M^-).\]
\end{lemma}

\begin{proof}
Let $N$ be a subposition
of $M$ such that $\epsilon_M(N) \ge e$
and let $N^\star = N \backslash \{ w, w' \}$.
We consider four possible cases for $N$.
\begin{itemize}
\item[(a)] If $w \in N$ and $w' \in N$ then $\wt{N} = \wt{N^\star \cup \{w, w'\}}=  \wt{N^\star \cup \{w + w'\}}$ and
$\epsilon_M(N) = \epsilon_{M^+}(N^\star \cup \{w + w'\})$.
\item[(b)] If $w \not\in N$ and $w' \not\in N$ then $\wt{N} = \wt{N^\star}$ and 
$\epsilon_M(N) = \epsilon_{M^+}(N^\star)$.
\item[(c)] If $w \in N$ and $w' \not\in N$ then $\wt{N} = \wt{N^\star \cup \{w\}} = \wt{N^\star \cup \{w - w'\}} + w'$
and $\epsilon_M(N) = \epsilon_{M^-}(N^\star \cup \{w-w'\})$.
\item[(d)] If $w \not\in N$ and $w' \in N$ then $\wt{N} = \wt{N^\star \cup \{w'\}} = \wt{N^\star} + w'$ and
$\epsilon_M(N) = \epsilon_{M^-}(N^\star)$.
\end{itemize}
Thus $\delta_e(M^+) = \sum_{N} (-1)^{\wt{N}}$ where the sum is over all subpositions
$N$ of $M$ such that $\epsilon_M(N) \ge e$ and either (a) or (b) holds, and
$(-1)^{w'} \delta_e(M^-) = \sum_{N} (-1)^{\wt{N}}$ where the sum is over all subpositions
$N$ of $M$ such that $\epsilon_M(N) \ge e$ and either (c) or (d) holds.
\end{proof}

We now define a family of further statistics.
For $e\in \N$ and
a position $M$ such that $\wt{M}$ and $e$ have the same parity, define
$\delta_e^{(1)}(M) = \delta_e(M)$, and for $b \in \N$ such that $b \ge 2$ define recursively
\[ \delta^{(b)}_e(M) = \sum_{t \in \N_0} \delta^{(b-1)}_{e+2t}(M). \]
For an alternative expression for $\delta^{(b)}_e(M)$ see Lemma~\ref{lemma:alt}.


We note that if $e + 2t > \wt{M}$ 
then $\delta^{(b-1)}_{e+2t}(M) = 0$,
and so the sum defining $\delta^{(b-1)}_e(M)$ is finite. Since we only need positions whose sum of elements
is at most $n$, it follows by induction on $b$ that
%
$\delta^{(b)}_e$ is a linear combination of the statistics $\delta_d$
for $d \ge e$. Hence, if $M$, $M^+$, $M^-$ and $w$, $w'$ 
are as in Lemma~\ref{lemma:conserved}, we have
\[ \delta^{(b)}_e(M) = \delta^{(b)}_e(M^+) + (-1)^{w'} \delta^{(b)}_e(M^-) 
\tag{$\star$}  . \]
For $r \in \N$ let $P(r)$ denote the highest power of $2$ dividing $r$
and let \hbox{$P(0) = \infty$}. 
Let
\[ \SW^{(b)}_e(M) = e + P\bigl( \delta^{(b)}_e(M)\bigr) \]
where, as expected, we set $e + \infty = \infty$. Our key statistic is now
\[ \SW_e(M) = \SW^{(e)}_e(M). \]
We remark that if $\wt{M}$ is odd then $\SW_1(M) = \Phi(M)$ where $\Phi(M)$
is as defined in \cite[page 386]{SaksWerman}.

In \S 5 below we prove Proposition~\ref{prop:bound} by
using $\SW_e(M)$ to prove an upper bound on the value $V(M)$ of a position $M$. 
The key properties of $\SW_e(M)$ we require are
$(\star)$ and the values of $\SW_e(M)$ at starting and final positions. Starting
positions are dealt with in the following lemma, whose proof uses the basic identity
$\sum_{r=0}^s (-1)^r \binom{n}{r} = (-1)^s \binom{n-1}{s}$; see \cite[Equation 5.16]{CMath}.

\begin{lemma}\label{lemma:start}
Let $M_{\mathrm{start}}$ be the multiset containing $1$ with multiplicity $n$.
Suppose that $n = 2s+e$ where $s$, $e\in \N$. Then for any $b \in \N$ such that $b \le n$ we have
\[ \delta_e^{(b)}(M_\mathrm{start}) = (-1)^s \binom{n-b}{s}.\]
\end{lemma}

\begin{proof}
When $b=1$ we have $\delta_e^{(1)} = \delta_e$.
A subposition $N$ of~$M_\mathrm{start}$
contributes to the sum defining $\delta_e(M)$ if and only if $\wt{N} \le s$.
Therefore
\[ \delta_e^{(1)}(M) = \sum_{r=0}^s (-1)^r \binom{n}{r} = (-1)^s \binom{n-1}{s}.\]
If $b \ge 2$ then, by induction, we have
\[ \delta_e^{(b)}(M) = \sum_{t \in \N_0} \delta_{e+2t}^{(b-1)}(M) =
\sum_{t=0}^s (-1)^{s-t} \binom{n-(b-1)}{s-t} = (-1)^s \binom{n-b}{s}
\]
again as required.
\end{proof}

It follows that if $M_\mathrm{start}$, $s$ and $e$ are as in Lemma~\ref{lemma:start}, then
$\SW_e(M_\mathrm{start}) = e + P \bigl( \binom{2s}{s} \bigr)$.
It is well known that
$P\bigl( \binom{2t}{t} \bigr) = B(t)$ for any $t \in \N$. (Two different
proofs are given in [2] and [4].) Hence 
\[ \tag{$\dagger$} \SW_e(M_\mathrm{start}) = e + B(s). \]

\section{Final positions}

Let $e = k - (n-k)$.
In this section we show that if $M$ is a final position containing exactly $c$ elements then
$\SW_e(M) \ge c$. The proof uses the \emph{hyperderivative} on the ring $\Z[x,x^{-1}]$ of integral
Laurent polynomials, defined on the monomial basis for $\Z[x,x^{-1}]$ by
\[ D^{(r)} x^p = \binom{p}{r} x^{p-r} \]
for $p \in \Z$ and $r \in \N_0$. (This extends the usual definition of the hyperderivative
for polynomial rings, given in \cite[page 303]{LidlNiederreiter}.) The key property we require is 
the following small generalization of \cite[Lemma~6.47]{LidlNiederreiter}.

\begin{lemma}\label{lemma:Leibnizrule}
Let $f,g \in \Z[x,x^{-1}]$ be Laurent polynomials. Let $r \in \N$. Then
\[ D^{(r)}(fg) = \sum_{t=0}^r D^{(t)}(f) D^{(r-t)}(g). \]
\end{lemma}

\begin{proof}
By bilinearity it is sufficient to prove the lemma when $f= x^p$ and $g= x^{q}$ where $p$, $q \in \Z$. In this case the lemma follows from
\[ \binom{p+q}{r} = \sum_{t=0}^r \binom{p}{t} \binom{q}{r-t} \]
which is the Chu--Vandermonde identity; see \cite[Equation 5.22]{CMath}. 
\end{proof}

We also need an alternative expression for $\delta^{(b)}_e(M)$.
Let $\alpha_r(M)$ be the number of subpositions $N$ of a position $M$ such that $\wt{N} = r$. 
The proof of the next lemma uses the basic identity 
$\sum_{d=r}^n \binom{d}{r} = \binom{n+1}{r+1}$; see \cite[Table~174]{CMath}. 

\begin{lemma}\label{lemma:alt}
Let $M$ be a position such that $\wt{M} = 2s + e$. Then
\[ \delta^{(b)}_e(M) = \sum_{r=0}^{s} \binom{s+b-1-r}{b-1} (-1)^{r} \alpha_r(M). \] 
\end{lemma}

\begin{proof}
When $b=1$, we have $\delta^{(1)}_e(M) = \delta_e(M) = \sum_{r=0}^s (-1)^r \alpha_r(M)$.
If $b \ge 2$, then by induction, we have
\begin{align*} 
\delta_e^{(b)}(M) &= \sum_{t \in \N_0} \delta_{e+2t}^{(b-1)}(M) \\
&= \sum_{t=0}^s \sum_{r=0}^{s-t} \binom{s-t+b-2-r}{b-2} (-1)^r \alpha_r(M) \\
&= \sum_{r=0}^s 
\sum_{t =0}^{s-r} \binom{s-t+b-2-r}{b-2}   
(-1)^r \alpha_r(M) \\
&= \sum_{r=0}^s  \binom{s+b-1-r}{b-1} (-1)^r \alpha_r(M)
\end{align*}
as claimed.
\end{proof}

Now let $M = \{w_1, \ldots, w_c\}$ be a final position in the $k$-majority game
where $w_1 \ge \ldots \ge w_c$. Let $\wt{M} = 2s+e$. 
As remarked in \S 2, we have
\[ w_1 \ge s+1.  \]
Let
\[ g = x^{s+e-1} (1+x^{-w_2}) \ldots (1+x^{-w_c}). \]
and note that, since $w_2 + \cdots + w_c \le s+e-1$,  $g$ is a polynomial. Let
\[ g = \sum_{r=0}^{s+e-1} \alpha'_r(M) x^{s+e-1-r}. \]

If $r \le s$ then a subposition $N$ of $M$ such that $\wt{N} = r$ cannot contain $w_1$. Therefore
$\alpha'_r(M) = \alpha_r(M)$ whenever $r \le s$.
It follows that
\[ (D^{(e-1)} g) (-1) = \sum_{r=0}^{s} \binom{s-r+e-1}{e-1} (-1)^{s-r} \alpha_r(M). \]
Hence, by Lemma~\ref{lemma:alt}, we have 
\[ (D^{(e-1)} g) (-1) =  \delta^{(e)}_e(M). \]
(The normal derivative would introduce an unwanted $(e-1)!$ at the point.)
It follows from Lemma~\ref{lemma:Leibnizrule}, applied to the original
definition of $g$, that $D^{(e-1)}(g)$ is
a linear combination, with coefficients in $\Z$, of polynomials of the form
\[ h =  x^{s+e-1-a_1}
\prod_{i \in A} x^{-w_i-a_i}
\prod_{j \in B} (1+x^{-w_j}) 
\]
where $A$ is a subset of $\{2,\ldots, n\}$, $B = \{2,\ldots, n\}\backslash A$,
and $a_1 + \sum_{i \in A} a_i = e-1$ where each summand is non-negative.
It is clear that $h(-1) = 0$ unless $w_j$ is even for all $j \in B$, in which
case $h(-1) = \pm 2^{|B|}$. Since $|B| \ge (c-1) - (e-1) = c-e$, it follows
that $P\bigl( h(-1) \bigr) \ge c-e$ for all such polynomials $h$. 
Hence
\[ \tag{$\ddagger$}
\SW_e(M) = e + P\bigl( \delta^{(e)}_e(M)\bigr) = e + P\bigl( (D^{(e-1)} g) (-1) \bigr)
\ge c \]
as claimed at the start of this section.

\section{Proof of Proposition~\ref{prop:bound}}

We are now ready to prove Proposition~\ref{prop:bound}.
Let $e = k - (n-k)$ be the minimum excess of the majority colour over the minority colour
in the $k$-majority game with $n$ balls.

Let $M$ be a position. Suppose that an optimal
play for the Selector is to choose $w$ and $w' \in M$.
Since the Assigner wishes to minimize the number of elements in the final position, we have
\[ V(M) = \min \bigl( V(M^+), V(M^-) \bigr) \]
where $M^+$ and $M^-$ are as defined in Lemma~\ref{lemma:conserved}.
If $x, y \in \Z$ then $P(x+y) \ge \min \bigl( P(x), P(y) \bigr)$.
Hence ($\star$) in \S 3 implies that
\[ \SW_e(M) \ge \min\bigl( \SW_e(M^+), \SW_e(M^-) \bigr). \]
By ($\ddagger$) at the end of \S 4, if $M$ 
is a final position containing $c$ elements then $\SW_e(M) \ge c$.
In this case $V(M) = c$, so we have $\SW_e(M) \ge V(M)$.
It therefore follows by induction that
\[ \SW_e(M) \ge V(M) \]
for all positions $M$. It was seen in $(\dagger)$ at the end of \S 3 that
if $M_\mathrm{start}$ is the starting position then
$\SW_e(M_{\mathrm{start}}) = B(n-k) + e$ and so
\[ B(n-k) + e \ge V(M_{\mathrm{start}})  \]
as required.

\section{Final remark}
We end by showing that the statistics $\SW_e(M)$ do not
predict all optimal moves for the Assigner. 
We need the following lemma in the case when $m$ is odd; it can
be proved
in a similar way to Lemma~\ref{lemma:start}.

\begin{lemma}\label{lemma:last}
For any $m \in \N$ we have 
\[ \SW_1\bigl( \{2,1^{2m-1}\} \bigr) = \begin{cases} 2  + B(m-1) + P(m-1) & \text{if $m$ is odd} \\
                                         2 +  B(m-1) - P(m) & \text{if $m$ is even.} 
                                         \end{cases}  \]                                         
\end{lemma}

Let $m \equiv 3$ mod $4$ and let $M = \{1^{2m+1}\}$ be the 
starting position in the majority game with $n=2m+1$ and $k=m+1$. 
The positions the Assigner can choose between on the first move in the game
are
$M^+ = \{2,1^{2m-1} \}$ and $M^- = \{1^{2m-1},0\}$.
By Lemma~\ref{lemma:last}, we have 
\[ \SW_1(M^+) = 2 + B(m-1) + P(m-1) = 3 + B(m-1) = 2 + B(m).\] 
It is clear that removing a zero element from a position decreases 
its $\SW_1$ statistic by~$1$, so by ($\dagger$) at the end of \S 3 we have
\[ \SW_1(M^-) = 1 + \SW_1( \{1^{2m-1} \}) = 1 + 1 + B(m-1) = 1 + B(m).\]
Using the $\SW_1$ statistic, the Assigner will therefore choose
$M^-$ on the first move.
 However Lemma~\ref{lemma:last} implies that
$\SW_1(\{2,1^{2m-3},0\}) = 1 + B(m)$, and we have already seen that
$\SW_1\bigl( \{1^{2m-1}\}\bigr) = B(m)$.
It follows that playing to $M^+$ is also an optimal
move for the Assigner.


%


\def\cprime{$'$} \def\Dbar{\leavevmode\lower.6ex\hbox to 0pt{\hskip-.23ex
  \accent"16\hss}D}
\providecommand{\bysame}{\leavevmode\hbox to3em{\hrulefill}\thinspace}
\renewcommand{\MR}[1]{\relax }
\providecommand{\MRhref}[2]{%
  \href{http://www.ams.org/mathscinet-getitem?mr=#1}{#2}
}
\providecommand{\href}[2]{#2}

\end{document}